\documentclass[final,1p,times]{elsarticle}

\usepackage{amssymb}
 \usepackage{amsthm}
\usepackage{amscd}
\usepackage{amsmath}
\usepackage{amsfonts}
\usepackage{amssymb}
\usepackage{graphicx}
\newtheorem{theorem}{Theorem}

\usepackage{mathrsfs}
\usepackage{titletoc}

\newenvironment{remark}[1][Remark]
           {\medbreak\noindent {\em #1. \enspace}}
           {\par \medbreak}

\makeatletter \@addtoreset{equation}{section} \makeatother

\newcommand\lam{\lambda}
\newcommand\na{\nabla}

\newcommand\Del{\Delta}

\newcommand\ra{\rightarrow}

\newcommand\Rc{\textup{Rc}}

\newcommand\SL{\text{SL}(2, \mathbb{R})}
\def\ddt{\frac{d}{dt}}

\journal{Manuscripta Mathematica}

\begin{document}

\begin{frontmatter}

\title{Eigenvalues of the Laplace operator with potential under the backward Ricci flow on  locally homogeneous 3-manifolds}

\author{Songbo Hou}
\ead{housb@cau.edu.cn}
\address{Department of Applied Mathematics, College of Science, China Agricultural University,  Beijing, 100083, P.R. China}
\author{Shusen Yang}
\ead{13651293036@163.com}
\address{Department of Applied Mathematics, College of Science, China Agricultural University,  Beijing, 100083, P.R. China}

\begin{abstract}
  Let $\lam(t)$ be  the first eigenvalue of $-\Del+aR\, (a>0)$ under the backward Ricci flow on  locally homogeneous 3-manifolds, where $R$ is the scalar curvature.
 In the  Bianchi case,  we get the upper and lower bounds of $\lam(t)$. In particular, we show that when the the backward Ricci flow converges to a sub-Riemannian geometry after a proper re-scaling,  $\lam^{+}(t)$ approaches zero, where $\lam^{+}(t)=\max\{\lam(t),0\}$.

\end{abstract}


\end{frontmatter}

\section{Introduction}

The research  on  the eigenvalues of  operators under  geometric flows has attracted many attentions. In a seminal paper \cite{PG02}, Perelman depicted the nondecreasing behavior of
the first eigenvalue of  $-\Del +R/4$   under the Ricci flow, where $-\Del$ denotes the Laplace-Beltrami operator, $R$  denotes the scalar curvature. As an application, he proved the non-existence of nontrivial steady or expanding breathers on closed
manifolds. Perelman's results were extended by Cao \cite{Cao07} to the operator $-\Delta+\frac{R}{2}$ on manifolds with nonnegative curvature
operator under the Ricci flow. Li \cite{JFL07} removed Cao's curvature assumption  and got the similar conclusion. One may refer to \cite{ Cao08, CHL12, Ma06} for more studies  about the operator $-\Delta +aR \,(a\geq 0)$ and refer to \cite{WWZ10, Wu11} for  the $p$-Laplace operator. In particular, the upper and lower bounds of eigenvalues were obtained by analyzing the evolution equation on  closed Riemann surfaces \cite{CHL12,WWZ10}.

There are also many results under  other geometric flows. In \cite{LJ13}, Li obtained the monotonicity of  eigenvalues under various
re-scaled versions of Ricci flow. For the normalized powers of the $m$th mean curvature flow, Zhao \cite{ZL} established the monotonicity of the first eigenvalue of $p$-Laplace operator under certain conditions. Under the harmonic-Ricci flow, Li \cite{LI} studied the the eigenvalues and entropies. Fang and Yang \cite{FY} investigated the  first eigenvalue of the operator $-\Del_{\phi}+\frac{R}{2}$, where $-\Del_{\phi}$ is the Witten-Laplacian and constructed monotonic quantities under the Yamabe flow. Along  the re-scaled List's extended Ricci flow, Huang and Li \cite{HL} considered monotonicity formulae of eigenvalues
of the Laplacian and entropies. For the monotonicity of eigenvalues and quantities along  the Ricci-Bourguignon flow, one may refer to \cite{FA, W}. Related results also include \cite{AA, CHZ, FXZ,FYZ,GPT, HK, Mao}.

In general, it's difficult to get the upper and lower bounds of eigenvalues of geometric operators under  flows. The obstacle is that the metric is variable, hence the classical methods, such as gradient estimates of eigenfunctions, heat kernel estimates et al., can not be employed directly.  While on closed surfaces, by controlling the scalar curvature, we  studied the derivative of the eigenvalue, then got the upper and lower bounds of eigenvalues  by integration \cite{CHL12}. But the approach is not available for high dimensional manifolds. In order to explore possibilities on locally homogeneous 3-manifolds, we developed a new method to estimate the eigenvalues by comparing components of the Ricci curvature \cite{Hou1,Hou}. Similar methods were also taken by Korouki and Razavi \cite{KR19} to study the eigenvalue of $-\Del-R$ under the Ricci flow. In this paper, we consider the eigenvalues  of the operator $-\Del+aR$ under the backward Ricci flow on  locally homogeneous 3-manifolds, where $a$ is a positive constant.

 There are nine classes of locally homogeneous 3-manifolds. These classes can be divided into two families. According the classification in \cite{IJ92}, $H(3)$, $H(2)\times \mathbb{R}$ and $\text{SO(3)}\times  \mathbb{R}$, where $H(n)$ means the group of isometries of hyperbolic $n$-space, belong to the first family.  Six other classes $\mathbb{R}^3$, $\text{SU(2)}$, $\SL$, $\text{Heisenberg}$, $E(1,1)$ (the group of isometries of the plane with flat Lorentz metric) and $E(2)$ (the group of isometries of the Euclidian plane) belong to the second family which is called the Bianchi case.

 In the Bianchi case, there is a Milnor frame $(X_{1},X_{2},X_{3})$ which can diagonalize the initial metric and the Ricci tensor. The property that the Ricci flow keeps the diagonalization enables us to write

$$g=A(t)\theta^1\otimes \theta^1+B(t)\theta^2\otimes \theta^2+C(t)\theta^3\otimes \theta^3, $$
where $(\theta^1,\theta^2,\theta^3)$ is the  frame dual to $(X_{1},X_{2},X_{3})$. Hence we reduce the Ricci flow to an ODE system in $(A,B,C)$. The following backward Ricci flow
\begin{equation}\frac{\partial g}{\partial t}=2\Rc-\frac{2r}{3}g,\,\,g(0)=g_0\end{equation} was studied in \cite{CGS, CS}. An interesting phenomenon is that after a proper re-scaling,  the flow converges uniformly to a sub-Remannian geometry in many cases. For more results about the Ricci flow on locally homogeneous manifolds,  we refer the reader to
\cite {Hou2,IJL, KM01} and so on.

We obtain the following theorem which extends the results in \cite{Hou1}.

\begin{theorem}
Let $(M, g(t)), t\in [0,T_+),$  be a solution to the backward Ricci flow in the Bianchi case where $T_+$  is the maximal existence time. Let
$\lam(t)$ be  the first eigenvalue of $-\Del+aR$, $a\geq 0$.  When the flow  converges to a sub-Riemannian geometry after a proper re-scaling,  $\lam^+(t)$ goes to zero  as $t\rightarrow T_+$, where  $\lam^{+}(t)=\max\{\lam(t), 0\}$.
\end{theorem}
Since the first eigenvalue of the Laplacian is nonnegative and the convergence is described in  \cite{Hou1}, we only  study the eigenvalue of the operator $-\Del+aR$, $a> 0$.
The remaining part of this paper is arranged as follows. In Section 2, we introduce an evolution equation of the eigenvalue under the backward Ricci flow.
Since the analysis on $\mathbb{R}^3$ is trivial, from Section 3 to Section 7,  we only investigate the behaviors of eigenvalues on Heisenberg,  $\text{SU}(2)$,  $E(1,1)$, $E(2)$ and $\SL$  case by case.

\section{Evolution equation}
In this section, we give an evolution equation of the eigenvalue which is important  in the subsequent analysis.
\begin{theorem}
Let $(M, g(t)), t\in [0,T_+)$  be a solution to the backward Ricci flow  on a locally homogeneous  3-manifold.
Denote by $\lam(t)$ the first eigenvalue of $-\Del+aR$, $a>0$ and by $u(x,t)$ the associated  positive eigenfunction with   $\int u^2(x,t)d\mu=1.$
Then under the the backward Ricci flow, there holds
\begin{align}\label{eve}\ddt\lam=\frac{2}{3}R\lam-2a|\Rc|^2-\int (2R_{ij}\na_iu\na_ju)d\mu.\end{align}
\end{theorem}

We omit the proof  as it is  similar to Lemma 3.1 in \cite {CHL12}.

\section{\text{Heisenberg}}
For  a  given metric $g_0$, there exists  a  Milnor frame such that
$$\left[X_2,X_3\right]=2X_1,\,\,\,\left[X_3,X_1\right]=0,\,\,\,\left[X_1,X_2\right]=0.$$ Let $A_0$, $B_0$, $C_0$ be the initial value of $A$, $B$, $C$ respectively.
We take  the normalization $A_0B_0C_0=4$. Following the calculations on page 171 in \cite{BLN}, one has
\begin{equation}\label{he1}
R_{11}=\frac{1}{2}A^{3},\,\,
R_{22}=-\frac{1}{2}A^{2}B,\,\,
R_{33}=-\frac{1}{2}A^{2}C,\,\,
R=-\frac{1}{2}A^{2}.
\end{equation}
Then  the backward Ricci flow equations become
$$
\left\{
\begin{aligned}
&\ddt A=\frac{4}{3}A^3,\\
&\ddt B=-\frac{2}{3}A^2B,\\
&\ddt C=-\frac{2}{3}A^2C.
\end{aligned}
\right.
$$
Solving these equations gives
\begin{equation}\label{he2}
\left\{
\begin{aligned}
&A=A_{0}\left(1+\frac{16}{3}R_{0}t\right)^{-1/2},\\
&B=B_{0}\left(1+\frac{16}{3}R_{0}t\right)^{1/4},\\
&C=C_{0}\left(1+\frac{16}{3}R_{0}t\right)^{1/4},
\end{aligned}\right.
\end{equation}
where $R_{0}=-\frac{1}{2}A_{0}^{2}$.
The re-scaled flow $\bar{g}(t)=(C_0/C(t))g(t)$ converges uniformly to a sub-Riemannian geometry.

 Hereafter we denote by $\tau$, $c_i$  variable constants which can be understood from the context.
\begin{theorem}
Let $\lam(t)$ be the  first eigenvalue of $-\Del+aR$.  Suppose $B_0\geq C_0$.  Then we get
$$\lam(\tau)+\frac{c_1A_0^3}{4\left(1-\frac{8A_0^2}{3}\tau\right)^{3/2}}-\frac{c_1A_0^3}{4\left(1-\frac{8A_0^2}{3}t\right)^{3/2}}\leq \lam(t)\leq \lam(\tau)\left( \frac{1-\frac{8A_0^2}{3}t}{1-\frac{8A_0^2}{3}\tau}\right)^{\frac{1}{8}}
e^{\frac{3B_0}{2}\left[\left(1-\frac{8A_0^2\tau}{3}\right)^{1/4}-\left(1-\frac{8A_0^2t}{3}\right)^{1/4}\right]}
$$ for $t\geq \tau$, where $\tau$ is a fixed time and $c_1$  is a positive constant. As a consequence,  $\lam^+(t)\rightarrow0$ as $t\rightarrow 3/(8A_0^2)$.
\end{theorem}
\begin{proof}
Suppose  $B_0\geq C_0$. we get $|\Rc|^2=\frac{3}{4}A^4$ by (\ref{he1}). Then (\ref{eve}) together with  (\ref{he1}) implies
\begin{equation}\label{he3}\frac{2}{3}R\lam -2R_{11}\lam-\frac {3a}{2}A^4+2aR_{11}R\leq\ddt \lam\leq \frac{2}{3}R\lam -2R_{22}\lam-\frac {3a}{2}A^4+2aR_{22}R.\end{equation}
Since $A\rightarrow +\infty$, $B\ra 0$ and $C\ra 0$ as $t\rightarrow\frac{3}{8A_0^2}$,  we have
$$ -\frac {3a}{2}A^4+2aR_{22}R= -\frac {3a}{2}A^4+\frac{a}{2}A^4 B <0$$ after a time $\tau$.
This leads to
$$\ddt \lam\leq \frac{2}{3}R\lam -2R_{22}\lam=\left(-\frac{1}{3}A^2+A^2B\right)\lam$$ for $t\geq\tau$.
Hence
$$\ddt\left(\lam e^{\int_{\tau}^t\left(\frac{1}{3}A^2-A^2B\right) dt}\right)\leq 0.$$
It follows from the integration that
\begin{equation}\label{he4}\lam(t)\leq \lam(\tau)\left( \frac{1-\frac{8A_0^2}{3}t}{1-\frac{8A_0^2}{3}\tau}\right)^{\frac{1}{8}}
e^{\frac{3B_0}{2}\left[\left(1-\frac{8A_0^2\tau}{3}\right)^{1/4}-\left(1-\frac{8A_0^2t}{3}\right)^{1/4}\right]}.\end{equation}
 According to  (\ref{he1}), (\ref{he2}), (\ref{he3}) and (\ref{he4}), we get
$$\ddt \lam\geq -c_1A^5.$$
Integration on both sides of the above inequality from $\tau$
to $t$ gives
$$\lam(t)\geq \lam(\tau)+\frac{c_1A_0^3}{4\left(1-\frac{8A_0^2}{3}\tau\right)^{3/2}}-\frac{c_1A_0^3}{4\left(1-\frac{8A_0^2}{3}t\right)^{3/2}}.$$
It can be easily checked that $\lam^{+}(t)=\max\{\lam(t),0\}$ goes to $0$ as $t$ goes to $\frac{3}{8A_0^2}$.
This finishes the proof.

\end{proof}

\section{$\text{SU}(2)$}

There exists a Milnor frame for a given metric $g_{0}$ such that
$$\left[X_2,X_3\right]=2X_1,\,\,\,\left[X_3,X_1\right]=2X_2,\,\,\,\left[X_1,X_2\right]=2X_3.$$
 Under the normalization $A_0B_0C_0=4$, we have
\begin{equation}\label{se1}
R_{11}=\frac{1}{2}A[A^{2}-(B-C)^{2}],\,
R_{22}=\frac{1}{2}B[B^{2}-(A-C)^{2}],\,
R_{33}=\frac{1}{2}C[C^{2}-(A-B)^{2}]\,
\end{equation}
and
\begin{equation}\label{se2}R=\frac{1}{2}[A^{2}-(B-C)^{2}]+\frac{1}{2}[B^{2}-(A-C)^{2}]+\frac{1}{2}[C^{2}-(A-B)^{2}].\end{equation}

We recall  Cao's results under the assumption $A_0\geq B_0\geq C_0$ in \cite{CS} .
\begin{theorem}\noindent

(1) If $A_0=B_0=C_0$, there holds  $T_+=\infty$ and $g(t)=g_0$.

(2) If $A_0=B_0>C_0$,  there holds $T_+=\infty$, $A=B>C$ and

$$A\sim\frac{8}{3}t,\;\;\;C\sim\frac{9}{16}t^{-2},$$ as $t$ goes to $T_+$.

(3) If $A_0>B_0\geq C_0$, there holds $T_+<\infty$, $A>B\geq C$ and

$$A\sim\frac{\sqrt{6}}{4}(T_+-t)^{-1/2},\,B\sim\eta_1(T_+-t)^{1/4},\,C\sim\eta_2(T_+-t)^{1/4},$$ where $\eta_1,\,\eta_2$ are two positive constants.

In the third case, $\bar{g}(t)=(B_0/B(t))g(t)$ converges to a sub-Riemannian geometry.
\end{theorem}
We get the following results.
\begin{theorem}
Let $\lam(t)$ be the  first eigenvalue of $-\Del+aR$. We have the following results.

(1) If $A_0=B_0=C_0$, then $\lam(t)=\lam(0)$.

(2) If $A_0=B_0>C_0$, then there holds
$$\lam(\tau)+\frac{c_2\lam(\tau)}{(1+c_1)\tau^{c_1}}(\tau^{1+c_1}-t^{1+c_1})\leq \lam(t)\leq \lam(\tau)\left(\frac{t}{\tau}\right)^{c_1}$$ if $t\geq \tau$, where $\tau$ is a fixed time, and $c_1$, $c_2$  are positive constants.

(3) If $A_0>B_0\geq C_0$, then there holds
$$\lam(\tau)+c_2\left[(T_+-\tau)^{-3/2}-(T_+-t)^{-3/2}\right]\leq\lam(t)\leq \lam(\tau)\left(\frac{T_+-t}{T_+-\tau}\right)^{c_1}$$ if $t\geq \tau$, where $\tau$ is a fixed time, and $c_1$, $c_2$  are positive constants. In this case, $\lam^+(t)\rightarrow 0$ as $t\rightarrow T_+$.
\end{theorem}
\begin{proof}
\noindent

 (1) If $A_0=B_0=C_0$, then $\lam(t)$ is a constant.

(2) If $A_0=B_0>C_0$, Theorem 4.2 in \cite{Hou1} implies
$$R_{11}=R_{22}>R_{33}>0$$
  after a time $\tau$.
Then it follows that
\begin{equation}\label{se3}\frac{2}{3}R\lam -2R_{11}\lam-2a|\Rc|^2+2aR_{11}R\leq\ddt \lam\leq \frac{2}{3}R\lam -2R_{33}\lam-2a|\Rc|^2+2aR_{33}R.\end{equation} By  employing  (\ref{se1}), (\ref{se2}) and the second item in Theorem 4,  we have
\begin{equation}\label{se4}2a|\Rc|^2=\frac{a}{2}\left[(2BC-C^2)^2+(2BC-C^2)^2+C^4\right]\sim c_1t^{-2},\end{equation}
\begin{equation}\label{se5}2aR_{11}R=\frac{aA}{2}(2BC-C^2)\left[ 2BC-C^2+2AC-C^2+C^2\right]\sim c_2t^{-1},\end{equation}
$$\frac{2}{3}R\lam -2R_{11}\lam=\left(\frac{4}{3}BC-\frac{1}{3}C^2-2ABC+AC^2\right)\lam\sim -8\lam,$$
$$\frac{2}{3}R\lam -2R_{33}\lam=\left(\frac{4}{3}BC-\frac{1}{3}C^2-C^3\right)\lam\sim 2t^{-1}\lam ,$$
\begin{equation}\label{se6}2aR_{33}R\sim c_3t^{-7}.\end{equation}
Hence by (\ref{se3}), (\ref{se4}) and (\ref{se6}), we get
$$\ddt \lam\leq \left(\frac{4}{3}BC-\frac{1}{3}C^2-C^3\right)\lam.$$
Denoting  $D=\frac{4}{3}BC-\frac{1}{3}C^2-C^3$,  we obtain
$$\ddt\left(\lam e^{\int_{\tau}^t-Ddt}\right)\leq 0$$ if $t
\geq \tau$.
It is easy to see that
\begin{equation}\label{se7}\lam(t)\leq \lam(\tau)\left(\frac{t}{\tau}\right)^{c_1}.\end{equation}
Hence  (\ref{se3}) together with (\ref{se4}), (\ref{se5}) and (\ref{se7}) leads to
$$-c_2\lam(\tau)\left(\frac{t}{\tau}\right)^{c_1}\leq\left(\frac{4}{3}BC-\frac{1}{3}C^2-2ABC+AC^2\right)\lam(\tau)\left(\frac{t}{\tau}\right)^{c_1}\leq \ddt\lam.$$
Integration from $\tau$  to $t$ gives $$\lam(\tau)+\frac{c_2\lam(\tau)}{(1+c_1)\tau^{c_1}}(\tau^{1+c_1}-t^{1+c_1})\leq \lam(t).$$
$$$$

(3) If  $A_0>B_0\geq C_0$, it follows from  Theorem 4.2 in \cite{Hou1} that
$$R_{11}>0,\,\,R_{22}<0,\,\,R_{33}<0$$
and $$R_{11}>R_{33}\geq R_{22}$$
after a time $\tau$.
Hence
\begin{equation}\label{se8}\frac{2}{3}R\lam -2R_{11}\lam-2a|\Rc|^2+2aR_{11}R\leq\ddt \lam\leq \frac{2}{3}R\lam -2R_{22}\lam-2a|\Rc|^2+2aR_{22}R.\end{equation}
Using (4.1), (4.2) and the third item in Theorem 4, we get

\begin{equation}\label{se9}2a|\Rc|^2=\frac{a}{2}\left\{\left[A^2-(B-C)^2\right]^2+\left[B^2-(A-C)^2\right]^2+\left[C^2-(A-B)^2\right]^2\right\}\sim c_1(T_+-t)^{-2},\end{equation}
\begin{equation}\label{se10} 2aR_{11}R=\frac{aA}{2}\left[A^2-(B-C)^2\right]\left(2AB+2AC+2BC-A^2-B^2-C^2\right)\sim-c_2(T_+-t)^{-5/2},\end{equation}
$$\frac{2}{3}R\lam -2R_{11}\lam\sim -c_3 (T_+-t)^{-3/2}\lam,$$
$$\frac{2}{3}R\lam -2R_{22}\lam\sim -c_4(T_+-t)^{-1}\lam,$$
\begin{equation}\label{se11} 2aR_{22}R\sim c_5(T_+-t)^{-7/4}.\end{equation}
By (\ref{se8}), (\ref{se9}) and (\ref{se11}), we obtain
$$\ddt\lam\leq \frac{2}{3}R\lam -2R_{22}\lam$$ if $t\geq \tau$.
Integrating from $\tau$ to $t$, we have
\begin{equation}\label{se12}\lam(t)\leq \lam(\tau) e^{\int_{\tau}^{t}\left(\frac{2}{3}R-2R_{22}\right)dt}\leq \lam(\tau)\left(\frac{T_+-t}{T_+-\tau}\right)^{c_1},\end{equation}
which together with (\ref{se8}), (\ref{se9}) and (\ref{se10}) implies
$$-c_2(T_+-t)^{-5/2}\leq \ddt\lam.$$
By integration, we obtain
$$\lam(\tau)+c_2\left[(T_+-\tau)^{-3/2}-(T_+-t)^{-3/2}\right]\leq \lam(t).$$
Finally, it is true that  $\lam^+(t)\rightarrow 0$ as $t\rightarrow T_+$.
\end{proof}

\section{$E(1,1)$}
Given a metric $g_0$,  we have  a fixed Milnor frame such that
$$\left[X_2,X_3\right]=2X_1,\,\,\,\left[X_3,X_1\right]=0\,\,\,\left[X_1,X_2\right]=-2X_3.$$
Choosing the normalization $A_0B_0C_0=4$, one has
\begin{equation}\label{ee1}
R_{11}=\frac{1}{2}A(A^2-C^2),\,
R_{22}=-\frac{1}{2}B(A+C)^2,\,
R_{33}=\frac{1}{2}C(C^2-A^2),\,
R=-\frac{1}{2}(A+C)^{2}.
\end{equation}

Suppose $A_0\geq C_0$. Cao \cite{CS} described  the following behaviors.
\begin{theorem}\noindent

(1) If $A_0=C_0$, then we have $T_+=\frac{3}{32}B_0$ and
$$A(t)=C(t)=\frac{\sqrt{6}}{4}(T_+-t)^{-1/2},\;\;B(t)=\frac{32}{3}(T_+-t),\;\;t\in[0,T_+).$$

(2) If $A_0>C_0$, then we have  $T_+<\infty$, and
$$A\sim\frac{\sqrt{6}}{4}(T_+-t)^{-1/2},\,B(t)\sim\eta_1(T_+-t)^{1/4},\,C(t)\sim\eta_2(T_+-t)^{1/4},$$
as $t$ goes to $T_+$, where $\eta_1,\eta_2$ are two positive constants.

In the second case, $\bar{g}(t)=(B_0/B(t))g(t)$ converges to a sub-Riemannian geometry.
\end{theorem}
We investigate the eigenvalue and get the following results.
\begin{theorem}
Let $\lam(t)$ be the  first eigenvalue of $-\Del+aR$. Then we get

(1) If $A_0=C_0$, then we have
$$\lam(\tau)+c_1\left[(T_+-\tau)^{-1}-(T_+-t)^{-1}\right]\leq \lam(t)\leq \lam(\tau)\left(\frac{T_+-t}{T_+-\tau}\right)^{1/2}e^{16(t-\tau)}$$ if $t\geq \tau$,
where $\tau$  is a fixed time and $c_1$ is a positive constant.

(2) If $A_0>C_0\geq B_0$, then for $t\geq \tau$,  there holds

$$\lam(\tau)+c_2\left[(T_+-\tau)^{-\frac{3}{2}}-(T_+-t)^{-\frac{3}{2}}\right]\leq \lam(t)\leq \lam(\tau)\left(\frac{T_+-t}{T_+-\tau}\right)^{c_1}.$$
In both cases, $\lam^+(t)\rightarrow 0$ as $t\rightarrow T_+$.
\end{theorem}
\begin{remark}
If $A_0>B_0> C_0$,  the analogous estimates hold.
\end{remark}
\begin{proof}
(1) If $A_0=C_0$,  then (5.1) together with  the first item in Theorem 6 implies

$$R_{11}=0,\,\,R_{22}<0,\,\,R_{33}=0$$
and  $$R_{11}=R_{33}>R_{22}.$$
Then we have \begin{equation}\label{ee2}\frac{2}{3}R\lam -2a|\Rc|^2\leq\ddt \lam\leq \frac{2}{3}R\lam -2R_{22}\lam-2a|\Rc|^2+2aR_{22}R.\end{equation}
There exist a time $\tau$ such that
\begin{equation}\label{ee3}2a|\Rc|^2=\frac{a}{2}\left[(A^2-C^2)^2+(A+C)^4+(C^2-A^2)^2\right]\sim c_1(T_+-t)^{-2},\end{equation}
\begin{equation}\label{ee4}2aR_{22}R=\frac{a}{2}B(A+C)^4\sim c_2(T_+-t)^{-1}\end{equation} if $t\geq \tau$.
Hence it follows from (\ref{ee2}), (\ref{ee3}) and (\ref{ee4}) that
$$\ddt\lam \leq \frac{2}{3}R\lam -2R_{22}\lam=\left(-\frac{1}{2}(T_+-t)^{-1}+16\right)\lam$$
for $t\geq \tau$.
This leads to
$$\ddt\left(\lam e^{\int_{\tau}^t\left(\frac{1}{2}(T_+-t)^{-1}-16\right)dt}\right)\leq 0.$$
Integration from $\tau$ to $t$ gives
\begin{equation}\label{ee5}\lam(t)\leq \lam(\tau)\left(\frac{T_+-t}{T_+-\tau}\right)^{1/2}e^{16(t-\tau)}.\end{equation}
By (\ref{ee2}), (\ref{ee3}) and (\ref{ee5}),  we get

 $$-c_1(T_+-t)^{-2}\leq \ddt \lam $$
 after a time $\tau$. It is concluded by integration that
 $$\lam(\tau)+c_1\left[(T_+-\tau)^{-1}-(T_+-t)^{-1}\right]\leq \lam(t).$$

(2)  If $A_0>C_0\geq B_0$, by Theorem 5.2 in \cite{Hou1} we have
$$R_{11}>0,\,\,R_{22}<0,\,\,R_{33}<0 $$
and $$R_{11}>R_{22}>R_{33}$$ after a time $\tau$.

 It follows that
 \begin{equation}\label{ee6}\frac{2}{3}R\lam -2R_{11}\lam-2a|\Rc|^2+2aR_{11}R\leq\ddt \lam\leq \frac{2}{3}R\lam -2R_{33}\lam-2a|\Rc|^2+2aR_{33}R.\end{equation}
Combining (5.1) and the second item in Theorem 6, we calculate
\begin{equation}\label{ee7}2a|\Rc|^2=\frac{a}{2}\left[(A^2-C^2)^2+(A+C)^4+(C^2-A^2)^2\right]\sim c_1(T_+-t)^{-2},\end{equation}
\begin{equation}\label{ee8}2aR_{33}R=\frac{a}{2}C(A+C)^2(A^2-C^2)\sim c_2(T_+-t)^{-\frac{7}{4}},\end{equation}
\begin{equation}\label{ee9}2aR_{11}R=-\frac{a}{2}A(A^2-C^2)(A+C)^2\sim -c_3(T_+-t)^{-\frac{5}{2}}.\end{equation}
By (\ref{ee6}), (\ref{ee7}) and (\ref{ee8}), we obtain
$$\ddt \lam\leq \frac{2}{3}R\lam -2R_{33}\lam$$
for $t\geq \tau$.

Integration from $\tau$ to $t$ yields
\begin{equation}\label{ee10}\lam(t)\leq \lam(\tau) e^{\int_{\tau}^t(\frac{2}{3}R-2R_{33})dt}\leq \lam(\tau)\left(\frac{T_+-t}{T_+-\tau}\right)^{c_1}.\end{equation}
It is clear that $\lam^+(t)\rightarrow 0$ as $t\rightarrow T_+$.  Using (\ref{ee6}), (\ref{ee7}), (\ref{ee9}) and (\ref{ee10}),  we immediately get
$$-c_2(T_+-t)^{-\frac{5}{2}}\leq \ddt\lam.$$
It follows from the integration that
$$\lam(\tau)+c_2\left[(T_+-\tau)^{-\frac{3}{2}}-(T_+-t)^{-\frac{3}{2}}\right]\leq \lam(t).$$

\end{proof}

\section{$E(2)$}
 There is a  Milnor frame for a metric $g_0$ such that
$$\left[X_2,X_3\right]=2X_1,\,\,\,\left[X_3,X_1\right]=2X_2,\,\,\,\left[X_1,X_2\right]=0.$$
Choosing the normalization $A_0B_0C_0=4$, then we have

\begin{equation}\label{es61}
R_{11}=\frac{1}{2}A(A^2-B^2),\,
R_{22}=\frac{1}{2}B(B^2-A^2),\,
R_{33}=-\frac{1}{2}C(A-B)^2,\,
R=-\frac{1}{2}(A-B)^{2}.
\end{equation}
Cao \cite{CS} gave  the following results.

\begin{theorem}\noindent

(1) If $A_0=B_0$, there holds  $T_+=\infty$, and $g(t)=g_0$.

(2) If $A_0>B_0$, then there holds $T_+<\infty$, and
$$A\sim\frac{\sqrt{6}}{4}(T_+-t)^{-1/2},\,B(t)\sim\eta_1(T_+-t)^{1/4},\,C(t)\sim\eta_2(T_+-t)^{1/4}$$

as $t$ goes to $T_+$, where $\eta_1,\,\eta_2$ are two positive constants.

In the second case, $\bar{g}(t)=(B_0/B(t))g(t)$ converges to a sub-Riemannian geometry.
\end{theorem}
We will prove the following theorem.
\begin{theorem}
Let $\lam(t)$ be the  first eigenvalue of $-\Del+aR$. Then we get

(1) If $A_0=B_0$, then $g(t)=g_0$, and $\lam(t)=\lam(0)$.

(2) If $A_0>B_0$ and $C_0\geq B_0$,  there holds that

$$\lam(\tau)+c_2\left[(T_+-\tau)^{-\frac{3}{2}}-(T_+-t)^{-\frac{3}{2}}\right]\leq \lam(t)\leq \lam(\tau)\left(\frac{T_+-t}{T_+-\tau}\right)^{c_1} $$ for $t\geq \tau$. In this case, $\lam^+(t)\rightarrow 0$ as $t\rightarrow T_+$
\end{theorem}
\begin{remark}
If $A_0>B_0$ and $C_0<B_0$, the similar estimates hold.
\end{remark}
\begin{proof}
\noindent

(1) If $A_0=B_0$, then $g(t)=g_0$, and $\lam(t)=\lam(0)$.

(2) If $A_0>B_0$, by Theorem 6.2 in \cite{Hou1}, we know that
$$R_{11}>0,\,\,R_{22}<0,\,\,R_{33}<0$$ and
$R_{11}>R_{22}>R_{33}$ with $t\geq \tau$ for some $\tau$.

It is easy to see that
\begin{equation}\label{es62}\frac{2}{3}R\lam -2R_{11}\lam-2a|\Rc|^2+2aR_{11}R\leq\ddt \lam\leq \frac{2}{3}R\lam -2R_{33}\lam-2a|\Rc|^2+2aR_{33}R\end{equation}for $t\geq \tau$.

By (\ref{es61}) and the second item in Theorem 8, we arrive at
\begin{equation}\label{es63}2a|\Rc|^2=\frac{a}{2}\left[ (A^2-B^2)^2+(B^2-A^2)^2+(A-B)^4\right]\sim c_1(T_+-t)^{-2},\end{equation}
\begin{equation}\label{es64}2aR_{33}R=\frac{a}{2}C(A-B)^4\sim c_2(T_+-t)^{-7/4},\end{equation}
\begin{equation}\label{es65}2aR_{11}R=-\frac{a}{2}A(A^2-B^2)(A-B)^2\sim -c_3(T_+-t)^{-5/2}.\end{equation}
Hence (\ref{es62}) together with (\ref{es63}) and (\ref{es64}) leads to
$$\ddt \lam\leq \frac{2}{3}R\lam -2R_{33}\lam.$$
By integration from $\tau$ to $t$, we get
\begin{equation}\label{es66}\lam(t)\leq \lam(\tau) e^{\int_{\tau}^t(\frac{2}{3}R-2R_{33})dt}\leq \lam(\tau)\left(\frac{T_+-t}{T_+-\tau}\right)^{c_1},\end{equation} which implies that
$\lam^+(t)\rightarrow 0$ as $t\rightarrow T_+$.
Observing (\ref{es62}), (\ref{es63}), (\ref{es65}) and (\ref{es66}),  we obtain

$$-c_2(T_+-t)^{-\frac{5}{2}}\leq \ddt\lam.$$
We conclude by  the integration that
$$\lam(\tau)+c_2\left[(T_+-\tau)^{-\frac{3}{2}}-(T_+-t)^{-\frac{3}{2}}\right]\leq \lam(t).$$

\end{proof}

\section{$\SL$}
This class is characterized by the Lie bracket of the Milnor frame:
$$\left[X_2,X_3\right]=-2X_1,\,\,\,\left[X_3,X_1\right]=2X_2,\,\,\,\left[X_1,X_2\right]=2X_3.$$
Under the normalization $A_0B_0C_0=4$, we have
\begin{equation}
R_{11}=\frac{1}{2}A[A^{2}-(B-C)^{2}],\,
R_{22}=\frac{1}{2}B[B^{2}-(A+C)^{2}],\,
R_{33}=\frac{1}{2}C[C^{2}-(A+B)^{2}]\end{equation}
and \begin{equation}R=\frac{1}{2}[A^{2}-(B-C)^{2}]+\frac{1}{2}[B^{2}-(A+C)^{2}]+\frac{1}{2}[C^{2}-(A+B)^{2}].\end{equation}

 Under the assumption  $B_0\geq C_0$,  Cao \cite{CGS, CS} proved the following theorem.
\begin{theorem}
The backward Ricci flow exists in a finite time  and has  the following asymptotic behaviors:

(1) If there exists a time $t_0$ such that $A(t_0)\geq B(t_0)$, there holds
$$A\sim\frac{\sqrt{6}}{4}(T_+-t)^{-1/2},\,B(t)\sim\eta_1(T_+-t)^{1/4},\,C(t)\sim\eta_2(T_+-t)^{1/4},$$ where $\eta_1$, $\eta_2$ are two positive constants.

(2) If there exists a time $t_0$ such that $A(t_0)\leq B(t_0)-C(t_0)$,  there holds
$$A\sim\eta_1(T_+-t)^{1/4},\,B(t)\sim\frac{\sqrt{6}}{4}(T_+-t)^{-1/2},\,C(t)\sim\eta_2(T_+-t)^{1/4}$$ with positive constants $\eta_1$ and $\eta_2$.

(3) If $A<B<A+C$ for all time $t\in [0,T_+)$, we arrive at
$$A\sim\frac{\sqrt{6}}{4}(T_+-t)^{-1/2},\,B(t)\sim\frac{\sqrt{6}}{4}(T_+-t)^{-1/2},\,C(t)\sim\frac{32}{3}(T_+-t).$$

In all cases, the metric $g(t)$ converges to a sub-Riemannian geometry after a proper re-scaling.
\end{theorem}
We have the following theorem.
\begin{theorem}
Let $\lam(t)$ be the  first eigenvalue of $-\Del+aR$. Then we get

(1) If there is a time $t_0$ such that $A(t_0)\geq B(t_0)$, then  there exists a  time $\tau$ such that
$$\lam(\tau)+c_2\left[(T_+-\tau)^{-\frac{3}{2}}-(T_+-t)^{-\frac{3}{2}}\right]\leq \lam(t)\leq  \lam(\tau)\left(\frac{T_+-t}{T_+-\tau}\right)^{c_1}$$ for $t\geq \tau$.

(2)  If there exist a time $t_0$ such that $A(t_0)\leq B(t_0)-C(t_0)$, and a time $t_1$ such that $A(t_1)>C(t_1)$, then there is a time $\tau$ such that
$$\lam(\tau)+c_2\left[(T_+-\tau)^{-1}-(T_+-t)^{-1}\right]\leq \lam(t)\leq  \lam(\tau)\left(\frac{T_+-t}{T_+-\tau}\right)^{c_1}$$ for $t\geq \tau$.

(3) If $ A<B < A + C$ for all time $t \in [0, T_+)$,  then there is time $\tau$ such that
$$\lam(\tau)+c_2\left[(T_+-\tau)^{-1}-(T_+-t)^{-1}\right]\leq \lam(t)\leq  \lam(\tau)\left(\frac{T_+-t}{T_+-\tau}\right)^{c_1}$$ for $t\geq \tau$.

In all cases, $\lam^+(t)\rightarrow 0$ as $t\rightarrow T_+$.
\end{theorem}
\begin{remark}
In the second item  of the above theorem,   an  analogous estimate holds if  $A(t)\leq C(t)$ for all $t$.
\end{remark}

\begin{proof}
(1) If there is a time $t_0$ such that $A(t_0)>B_{0}$, then
we know from Theorem 7.2 in \cite{Hou1} that
$$R_{11}>0,\,\,R_{22}<0,\,\,R_{33}<0$$
and  $$R_{11}>R_{33}\geq R_{22}$$  after a time $\tau$.

Hence
\begin{equation}\label{es71}\frac{2}{3}R\lam -2R_{11}\lam-2a|\Rc|^2+2aR_{11}R\leq\ddt \lam\leq \frac{2}{3}R\lam -2R_{22}\lam-2a|\Rc|^2+2aR_{22}R\end{equation}for $t\geq \tau$.
It follows from (7.1), (7.2)  and the first item in Theorem 10 that
\begin{equation}\label{es72}2a|\Rc|^2=\frac{a}{2}[A^{2}-(B-C)^{2}]^2+\frac{a}{2}[B^{2}-(A+C)^{2}]^2+\frac{a}{2}[C^{2}-(A+B)^{2}]^2\sim c_1(T_+-t)^{-2},\end{equation}
\begin{equation}\label{es73}2aR_{22}R=\frac{a}{2}B[B^{2}-(A+C)^{2}]\left[2BC-2AC-2AB-A^2-B^2-C^2\right]\sim c_2(T_+-t)^{-7/4},\end{equation}
\begin{equation}\label{es74}2aR_{11}R=\frac{a}{2}A[A^2-(B-C)^2]\left[2BC-2AC-2AB-A^2-B^2-C^2\right]\sim -c_3(T_+-t)^{-5/2}.\end{equation}
Then by (\ref{es71}), (\ref{es72}) and (\ref{es73}), we arrive at
$$\ddt \lam\leq \frac{2}{3}R\lam -2R_{22}\lam$$ if $t\geq \tau$.
Integration from $\tau$ to $t$ gives
\begin{equation}\label{es75}\lam(t)\leq \lam(\tau) e^{\int_{\tau}^t(\frac{2}{3}R-2R_{22})dt}\leq \lam(\tau)\left(\frac{T_+-t}{T_+-\tau}\right)^{c_1}.\end{equation}
Thus,  (\ref{es71}) together with (\ref{es72}), (\ref{es74}), (\ref{es75}) implies $$-c_2(T_+-t)^{-\frac{5}{2}}\leq \ddt\lam.$$
It is  concluded by integration that
$$\lam(\tau)+c_2\left[(T_+-\tau)^{-\frac{3}{2}}-(T_+-t)^{-\frac{3}{2}}\right]\leq \lam(t).$$

(2) If there exists a time $t_0$ such that $A(t_0)\leq B(t_0)-C(t_0)$ and a time $t_1$ such that $A(t_1)>C(t_1)$,
then by Theorem 7.2 in \cite{Hou1}, there  holds
$$R_{11}<0,\,\,R_{22}>0,\,\,R_{33}<0$$ and

$$R_{22}>R_{33}>R_{11}$$ after a time $\tau$.

We obtain
$$\frac{2}{3}R\lam -2R_{22}\lam-2a|\Rc|^2+2aR_{22}R\leq\ddt \lam\leq \frac{2}{3}R\lam -2R_{11}\lam-2a|\Rc|^2+2aR_{11}R$$ for $t\geq \tau$.

Calculate
$$2a|\Rc|^2\sim c_1(T_+-t)^{-2},$$
$$2aR_{11}R=\frac{a}{2}A[A^2-(B-C)^2]\left[2BC-2AC-2AB-A^2-B^2-C^2\right]\sim c_2(T_+-t)^{-7/4},$$
$$2aR_{22}R=\frac{a}{2}B[B^{2}-(A+C)^{2}]\left[2BC-2AC-2AB-A^2-B^2-C^2\right]\sim-c_3(T_+-t)^{-5/2}.$$
Using the similar arguments as in the first case, we get
$$\lam(\tau)+c_2\left[(T_+-\tau)^{-1}-(T_+-t)^{-1}\right]\leq \lam(t)\leq \lam(\tau)\left(\frac{T_+-t}{T_+-\tau}\right)^{c_1}$$ for $t\geq \tau$.

(3) If $A<B<A+C$ for all time $t\in [0,T_+)$, then Theorem 7.2 in \cite{Hou1} implies
 $$R_{11}>0,\,\,R_{22}<0,\,\,R_{33}<0$$
 and  $$R_{11}>R_{22}\geq R_{33}$$ after  a time $\tau$.
 Consequently,
$$\frac{2}{3}R\lam -2R_{11}\lam-2a|\Rc|^2+2aR_{11}R\leq\ddt \lam\leq \frac{2}{3}R\lam -2R_{33}\lam-2a|\Rc|^2+2aR_{33}R$$ for $t\geq \tau$.

Direct calculations  give
$$2a|\Rc|^2\sim c_1(T_+-t)^{-2},$$
$$2aR_{11}R=\frac{a}{2}A[A^2-(B-C)^2]\left[2BC-2AC-2AB-A^2-B^2-C^2\right]\sim -c_2(T_+-t)^{-1},$$
$$2aR_{33}R=\frac{a}{2}C[C^{2}-(A+B)^{2}]\left[2BC-2AC-2AB-A^2-B^2-C^2\right]\sim c_3(T_+-t)^{-7/4}.$$
Proceeding as in the proof of the first case, we have
$$\lam(\tau)+c_2\left[(T_+-\tau)^{-1}-(T_+-t)^{-1}\right]\leq \lam(t)\leq \lam(\tau)\left(\frac{T_+-t}{T_+-\tau}\right)^{c_1}$$ for $t\geq \tau$.

\end{proof}
\vskip 30 pt
\noindent{\bf Acknowledgement}

This work is  partially  supported by the National Natural Science Foundation of China (Grant No. 11721101), and by National Key Research and Development Project SQ2020YFA070080.

\end{document}